\documentclass[11pt]{amsart}

\usepackage{amsmath,amscd,amssymb,latexsym,color}
\usepackage{longtable}
\usepackage{systeme}
\sysdelim..

\usepackage[makeroom]{cancel}

\newcommand{\sm}{\left(\smallmatrix}
\newcommand{\esm}{\endsmallmatrix\right)}
\newcommand{\mat}{\begin{pmatrix}}
\newcommand{\emat}{\end{pmatrix}}
\renewcommand{\c}{\mathfrak{c}}
\newcommand{\f}{\mathfrak{f}}
\newcommand{\fg}{\mathfrak{g}}
\renewcommand{\t}{\tau}

\renewcommand{\a}{\alpha}

\renewcommand{\i}{\infty}
\newcommand{\G}{\Gamma}
\newcommand{\g}{\gamma}
\newcommand{\ve}{\varepsilon}

\newcommand{\Q}{\mathbb Q}
\newcommand{\Z}{\mathbb Z}

\newcommand{\R}{\mathbb R}

\renewcommand{\H}{\mathbb H}

\newcommand{\F}{\mathcal F}

\newtheorem{thm}{Theorem}
\newtheorem{lem}[thm]{Lemma}
\newtheorem{cor}[thm]{Corollary}
\newtheorem{prop}[thm]{Proposition}
\newtheorem{exam}[thm]{Example}

\theoremstyle{definition}

\newtheorem{rmk}[thm]{Remark}
\numberwithin{equation}{section}
\numberwithin{thm}{section}

\begin{document}

\title[Values at Divisors]{On values of weakly holomorphic modular functions\\ at divisors of meromorphic modular forms}

\author{Daeyeol Jeon, Soon-Yi Kang and Chang Heon Kim}
\address{Department of Mathematics Education, Kongju National University, Kongju, 32588 Korea}
\email{dyjeon@kongju.ac.kr}
\address{Department of Mathematics, Kangwon National University, Chuncheon, 24341 Korea} \email{sy2kang@kangwon.ac.kr}
\address{Department of Mathematics, Sungkyunkwan University,
 Suwon, 16419 Korea}
\email{chhkim@skku.edu}

\begin{abstract} We show that the values of a certain family of weakly holomorphic modular functions at points in the divisors of any meromorphic modular form  with algebraic Fourier coefficients are algebraic. We use this to extend the classical result of Schneider by proving that zeros or poles of any non-zero meromorphic modular form  with algebraic Fourier coefficients are either transcendental or imaginary quadratic irrational. 
\end{abstract}
\maketitle

\renewcommand{\thefootnote}%
             {}
 {\footnotetext{

The first author was supported by National Research Foundation of Korea (NRF)
funded by the Ministry of Education (NRF-2016R1D1A1B03934504), the second author
was supported by the National Research Foundation of Korea (NRF) 
funded by the Ministry of Education (NRF-2016R1D1A1B01012258) 
and the third author was supported by National Research
Foundation of Korea (NRF) funded by the Ministry of Education (NRF-2021R1A2C1003998).
}

\section{Introduction}

Throughout the paper, we call a modular form, which is holomorphic away from the infinity, weakly holomorphic and write $M_k^\sharp(N)$ to denote the space of the infinite dimensional complex vector space of weakly holomorphic modular forms of weight $k$ for $\G_0(N)$.  Then the modular $j$-function defined for $\t$ in the complex upper half-plane $\H$ by
$$j(\t):=q^{-1}+744+196884q+21493760q^2+\cdots,\quad (q=e^{2\pi i \t})$$
 is a weakly holomorphic modular function on the full modular group $\G(1)=SL_2(\mathbb Z)$. It follows from the theory of complex multiplication that if $\t$ is an imaginary quadratic irrational number, then $j(\t)$ is an algebraic integer, known as {\it singular moduli}.  In 1937, Schneider \cite{Schneider} showed that the values of the $j$-function at algebraic number of degree                                                                                                                                                                                                                                             larger than 2 are transcendental.

 For the normalized Hecke operator $T(n)$, the weakly holomorphic modular functions given by
\begin{equation}\label{hj}J_{n}(\t):=(j(\t)-744)|T(n), \quad (n\geq 1)
\end{equation}
and $J_0(\t):=1$ form a basis of $M_0^\sharp(1)$.
 In \cite[Corollary 2]{BKO}, Bruinier, Kohnen and Ono proved that the sum of the values of $J_n$ at the points in the divisor of an algebraic modular form on $\G(1)$ are algebraic, and also the values of $j$ at those points are algebraic. Therefore, the result of Schneider leads us to conclude that zeros and poles of an algebraic modular form of level 1 are transcendental or imaginary quadratic irrational. The arguments in \cite{BKO} is constructive in that it proposes a method for obtaining the minimal polynomials of the algebraic values of $j$. Gun, Murty and Rath \cite[Theorem 6]{GMR} later showed that zeros of an algebraic modular form of any level are transcendental or imaginary quadratic irrational. These are great generalizations of the results on zeros of an algebraic modular form of level 1 that we recall below.

In \cite{RSD}, Rankin and Swinnerton-Dyer proved that all the zeros of the weight $k$ Eisenstein series $E_k(\t)$ ($k\geq 4$ even) are simple and  lie on the unit circle, the lower left arc of the standard fundamental domain $\F$ for the action of $\G(1)$ on $\H$. 
Using this elegant result of Rankin and Swinnerton-Dyer along with the result of Schneider and the theory of complex multiplication, Kohnen \cite{Kohnen} proved that the zeros of both $E_k(\t)$ and $J_n(\t)$ in $\F$, other than $i$ or $\rho:=e^{2\pi i/3}$, are transcendental. 

Duke and Jenkins \cite{DJ} also studied the location of zeros of a family of weakly holomorphic modular forms for $\G(1)$.  For even weight $k=12\ell+k'$, where $k'\in\{0,4,6,8,10,14\}$, they constructed a basis $$\{f_{k,m}=q^{-m}+O(q^{\ell+1})|m\geq -\ell\}$$ of $M_k^\sharp(1)$ and showed that when $m\geq|\ell|-\ell$, the zeros of $f_{k,m}$ all lie on the unit circle. Shortly thereafter, applying Kohnen's method, Jennings-Shaffer and Swisher \cite{JSS} proved that these zeros, other than $i$ and $\rho$, are  transcendental.

In this paper, we investigate the algebraic nature of divisors of an algebraic modular form on $\G_0(N)$ for any given level $N$. 
In \cite{JKK-hecke}, the authors constructed a {\it reduced row echelon basis} for $M_0^\sharp(N)$.  Assume, for simplicity, $\i$ is not a Weierstrass point of $X_0(N)$. Then if
 $m\geq g+1$, when $g:=g(N)$ is the genus of $X_0(N)$ with $g>0$, there is always a weakly holomorphic modular function that has only pole at $\i$ of order $m$ by the Weierstrass gap theorem.  By performing Gauss elimination on the coefficients of these functions, we can obtain the unique modular function of the form 
\begin{equation*}\label{mfg}\f_{N,m}(\t)=q^{-m}+\sum_{\ell=1}^{g} a_{N}(m,-\ell)q^{-\ell}+O(q)
\end{equation*}
for each $m\geq g+1$. 
Taking account of Yang's construction of the two generators of the function fields of the modular curves $X_0(N)$ in \cite{Yang}, which can be represented in terms of Dedekind eta-functions or generalized Dedekind eta-functions, we can assure that the coefficients of $\f_{N,m}$ are all algebraic, and most likely rational.
Furthermore, for each $m$, $\f_{N,m}=Q_m(\f_{N,g+1},\f_{N,g+2},\dots,\f_{N,2g+1})$ for some polynomial $Q_m$ in $X_1,\dots,X_{g+1}$ with  algebraic coefficients. %
When $g=0$, obviously, we take $\f_{N,1}$ as the Hauptmodul of $\G_0(N)$ and $\f_{N,m}$ as the Faber polynomial in $\f_{N,1}$.
Suppose the basis elements have the Fourier expansion
\begin{equation}\label{mfg}\mathfrak{f}_{N,m}(\t)=q^{-m}+\sum_{\ell=1}^{g} a_{N}(m,-\ell)q^{-\ell}+\sum_{n=1}^{\i} a_{N}(m,n)q^{n},\ (m\geq g+1).
\end{equation}
In addition, $\f_{N,0}=1$ and $\f_{N,m}=0$ for $1\leq m\leq g$ when the genus $g$ is positive. Of course,  $\f_{1,m}(\t)=J_{m}(\t)$. 
We also constructed a basis for $S_2^{\sharp}(N)$, the subspace of $M_2^{\sharp}(N)$ consisting of forms that vanish at all cusps except for the infinity.  We write for each $n\geq -g$
\begin{equation}\label{gF}
\fg_{N,n}(\t)=q^{-n}+\sum_{m\geq g+1} b_{N}(n,m)q^m. \end{equation}
The Fourier coefficients of $\f_{N,m}(\t)$ and $\fg_{N,n}(\t)$ satisfy the following duality condition (\cite[Theorem 1.12]{JKK-hecke}.) 
\begin{equation}\label{dual}
a_{N}(m,n)=-b_{N}(n,m), \quad (m\geq g+1, n\geq -g).
\end{equation}

Following the method in \cite{BKO}, we show that the values of $\f_{N,m}$ at the points in the divisor of an algebraic modular form on $\G_0(N)$ are algebraic, from which we infer that zeros and poles of an algebraic modular form of arbitrary level $N$ are transcendental or imaginary quadratic irrational. 

 In the next section, we introduce some preliminaries and state our results. Then we prove the main theorems through the following three sections and provide some examples in the last section. One example illustrates how to compute a minimal polynomial of the algebraic values of $\f_{11,2}$ at zeros of a function related with the elliptic genus of $K3$. Another example  shows the values of $\f_{27,n}$ for all $n\geq 2$ at the zeros of the function $\frac{\eta(3\t)^3}{\eta(27\t)^3} +3$, where $\eta$ is the Dedekind eta-function,  belong to the cubic field $\Q(\sqrt[3]{9})$. The last example presents an application of a main theorem to finding an explicit representation of traces of singular moduli of $\f_{N,n}$ for certain $N$'s in terms of exponents appearing in the product representation of a meromorphic modular function. 

\section{Statement of Results}

For even integer $k\geq 4$, the normalized Eisenstein series of weight $k$ 
$$E_k(\t):=1-\frac{2k}{B_k}\sum_{n=1}\sigma_{k-1}(n)q^n$$
is a weight $k$ modular form for $\G(1)$. Here, $B_k$ is the usual $k$th Bernoulli number and $\sigma_{k-1}(n):=\sum_{d|n} d^{k-1}$.   The weight 2 Eisenstein series
$$E_2(\t):=1-24\sum_{n=1}\sigma_{1}(n)q^n$$ 
is not a modular form. However,  
$\displaystyle{E_2(\t)-\frac{3}{\pi \textrm{Im}(\t)}}$ 
transforms like a modular form of weight $2$. 
Among the most important results on Eisenstein series established by Ramanujan are his differential equations 
$$\Theta(E_2)=\frac{1}{12}(E^2_2-E_4),\ 
\Theta(E_4)=\frac{1}{3}(E_2E_4-E_6),\ 
\Theta(E_6)=\frac{1}{2}(E_2E_6-E^2_4),
$$
where $\Theta:=q\frac{dq}{q}=\frac{1}{2\pi i}\frac{d}{d\t}$ is the Ramanujan's differential operator.
It satisfies
$$\Theta\left(\sum_{n=h}^\i a(n)q^n\right)=\sum_{n=h}^\i na(n)q^n.$$
For a meromorphic modular form $f$ of weight $k$ on $\G_0(N)$, $\Theta(f)$ fails to be modular, but the Serre derivative $\displaystyle{\partial_k(f):=\Theta(f)-\frac{kE_2f}{12}}$ preserves modularity with weight raised by $2$.
Hence  $\displaystyle{\frac{\partial_k(f)}{f}}$
is modular of weight 2.  
Bruinier, Kohnen and Ono \cite[Theorem 1]{BKO} proved that if 
\begin{equation}\label{f}
f(\t)=q^h+\sum_{n=h+1}^\i a_f(n)q^n
\end{equation} is a non-zero weight $k$ meromorphic modular form on $\G(1)$, then
\begin{equation}\label{18} \sum_{z\in \F_1}e_{1,z}\hbox{ord}_{z}(f)\sum_{n=0}^\i J_{n}(z) q^n=-\frac{\partial_k(f)}{f}.
\end{equation} 
Here, $\F_N$ is the fundamental domain of $\G_0(N)$ and $1/e_{N,z}$ is the cardinality of $
\G_0(N)_{z}/\{\pm1\}$ for each $z \in \H$, where $\G_0(N)_{z}$ denotes the stabilizer of $z$ in $\G_0(N)$. Note that if we take $f(\t)=j(\t)-j(z)$ for some $z\in\F_1$ in \eqref{18}, then we have the Asai-Kaneko-Ninomiya's identity \cite[Theorem 3]{AKN}:
$$\sum_{n=0}^\i J_{n}(z) q^n=\frac{1728 E_4^2(\t)E_6(\t)}{E_4^3(\t)-E_6^2(\t)}\cdot\frac{1}{j(\t)-j(z)}.$$

Bruinier, Kohnen and Ono \cite[Proposition 2.1]{BKO}  also proved that any meromorphic function in a neighborhood of $q=0$ with the Fourier expansion in \eqref{f} has an infinite product expansion 
\begin{equation}\label{pr}
f=q^h\prod_{n=1}^\i(1-q^n)^{c(n)}
\end{equation}
for uniquely determined complex numbers $c(n)$ and
\begin{equation}\label{divexp}
\frac{\Theta(f)}{f}=h-\sum_{n=1}\sum_{d|n}c(d) dq^n.
\end{equation}
Hence
\begin{equation}\label{serco}
\frac{\partial_k(f)}{f}=h-\frac{k}{12}+\sum_{n=1}^\i (-\sum_{d|n}c(d)d+2k\sigma_1(n))q^n.
\end{equation}
Moreover, they proved \cite[Theorem 5]{BKO} that for each $ n\in\mathbb{N}$, it satisfies that 
\begin{equation}\label{divexp1}
\sum_{z\in \F_1}e_{1,z}\hbox{ord}_{z}(f)J_{n}(z) =\sum_{d|n}c(d) d-2k\sigma_1(n).
\end{equation}

These results have been generalized by Bringmann et al. \cite[Theorem 1.3]{BKLOR} and Choi, Lee and Lim \cite[Theorem 1.1]{CLL}  to Niebur-Poincar\'e  harmonic weak Maass functions $J_{N,n}(\t)$ of arbitrary level $N$. 
%
For every point $z\in\H$, we define
\begin{align*}
{H}_{N,z}(\t):=&\sum_{n=0}^\i J_{N,n}(z)q^n,\cr
{F}_{N,z}(\t):=&\sum_{n=0}^\i \f_{N,n}(z)q^n=1+\sum_{n=g+1}^\i \f_{N,n}(z)q^n.\cr
\end{align*}

In \cite[Theorem 1.1]{CLL}, it is shown that if $f$ is a meromorphic modular function on $\G_0(N)$ with the product expansion as \eqref{pr}, then it satisfies 
\begin{equation}\label{cll10}
\sum_{z\in\F_N}e_{N,z}\hbox{ord}_{z}(f)H_{N,z}(\t)=-\frac{\partial_k(f)}{f}+{E}_{f}(\t),
\end{equation}
where ${E}_{f}(\t)$ is a modular form in the Eisenstein space of weight 2 on $\G_0(N)$.
Letting ${E}_{f}(\t)=\sum_{n=0}\ve_{f}(n)q^n$ and comparing coefficients of $q^n$ on each side of \eqref{cll10}, we find a generalization of \eqref{divexp1}: 
\begin{equation}\label{divexpN}
\sum_{z\in\F_N}e_{N,z}\hbox{ord}_{z}(f)J_{N,n}(z)=\sum_{d|n}c(d) d-2k\sigma_1(n)+\ve_f(n),\quad (n\in\mathbb N).
\end{equation}

One can establish similar results for $F_{N,z}(\t)$ and $\f_{N,n}(\t)$.

\begin{thm}\label{d} Let $g$ be the genus of $X_0(N)$ and $\i$ be not a Weierstrass point of $X_0(N)$. Assume 
$$ f(\t)=q^h+\sum_{n=h+1}^\i a_f(n)q^n=q^h\prod_{n=1}^\i(1-q^n)^{c(n)}$$
is a meromorphic modular form of weight $k$ on $\G_0(N)$ and 
 ${E}_{f}(\t)=\sum_{n=0}\ve_{f}(n)q^n$ is a modular form in the Eisenstein space of weight 2 on $\G_0(N)$ whose constant coefficient is identical to that of $\frac{\partial_k(f)}{f}$ at each cusp of  $\G_0(N)$ except for the infinity cusp. Then  for all $\t\in\H$,
\begin{equation}\label{gtfne}
\sum_{z\in \F_N}e_{N,z}\hbox{ord}_{z}(f){F}_{N,z}(\t)=-\frac{\partial_k(f)}{f}+{E}_{f}(\t)+\sum_{\ell=1}^g (-\sum_{d|\ell}c(d) d+2k\sigma_1(\ell)-\ve_{f}(\ell) ) \fg_{N,-\ell}(\t).
\end{equation}
Hence, for each $n\geq g+1$, it satisfies that
\begin{align}\label{divf}
\sum_{z\in \F_N}e_{N,z}\hbox{ord}_{z}(f)\f_{N,n}(z)
&=\sum_{d|n}c(d) d-2k\sigma_1(n)+\ve_{f}(n)\cr
&\qquad +\sum_{\ell=1}^g (\sum_{d|\ell}c(d) d-2k\sigma_1(\ell)+\ve_{f}(\ell) ) a_{N}(n,-\ell).
\end{align}
\end{thm}
\begin{rmk}\label{r1}
The spaces covered in \cite{JKK-hecke} are those of weakly holomorphic modular functions on subgroups of ${\rm PSL}_2(\R)$ generated by $\G_0(N)$ and some Atkin-Lehner involutions. The method of construction of bases of the spaces applies to any subgroup $\G$, where $X(\G)$ is a compact Riemann surface. Similar results with equations \eqref{gtfne} and \eqref{divf} in Theorem \ref{d} hold for such bases as well.
\end{rmk}

As an application of Theorem \ref{d}, we determine the algebraicity of the values of $\f_{N,n}(\t)$ at the divisors of meromorphic modular forms on $\G_0(N)$ in the fundamental domain $\F_N$. 
For the purpose, we first need to establish the algebraicity of Fourier coefficients of 
${E}_{f}(\t)$ in Theorem \ref{d}.

\begin{thm}\label{eis}
For any meromorphic modular form $f$, the Fourier coefficients of ${E}_{f}(\t)$ are algebraic. 
\end{thm}

Theorem \ref{d} and Theorem \ref{eis} together imply the sum of the values of $ \f_{N,n}$ at points of divisors of an algebraic modular form is algebraic. In addition, we show that the values of $\f_{N,n}$ is algebraic at those points. 

\begin{thm}\label{alg}  Let $ f(\t)$ be a meromorphic modular form on $\G_0(N)$ with algebraic Fourier coefficients. If 
$z_0\in\F_N$ is a point for which $\textrm{ord}_{z_0}(f)\neq 0$, then $\f_{N,n}(z_0)$ for all $n$ are algebraic.
\end{thm}

Assume all $\f_{N,g+1}(z_0), \f_{N,g+2}(z_0),\dots, \f_{N,2g+1}(z_0) $ are algebraic. Then since $j=\frac{P(\f_{N,g+1},\dots,\f_{N,2g+1})}{Q(\f_{N,g+1},\dots,\f_{N,2g+1})}$ for some polynomials $P$ and $Q$ with algebraic coefficients, $j(z_0)$ is algebraic. Therefore, by Schneider's theorem, $z_0$ must be  either transcendental or imaginary quadratic. 
\begin{cor}
Any zero or a pole of a meromorphic modular form of arbitrary weight and level with algebraic Fourier coefficients is  either  transcendental or imaginary quadratic. 
\end{cor}

\section{Proof of Theorem \ref{d}}

For each cusp $\mathfrak s$, we take $\gamma_{\mathfrak s}=\left(\begin{smallmatrix}
  a&b\\c &d
  \end{smallmatrix}\right) \in SL_{2}(\Bbb{Z})$ satisfying $\gamma_{\mathfrak s}( \infty) =
  \mathfrak s$.
Then there exists a unique positive real number $h_{\mathfrak s}$, so called {\it width} of the cusp $\mathfrak s$ such that
$$
\gamma_{\mathfrak s}^{-1}\G_0(N)_{\mathfrak s}\gamma_{\mathfrak s}=\{\pm
\left(\begin{smallmatrix}
  1&h_{\mathfrak s}\\0 &1
  \end{smallmatrix}\right)^{m} | m\in \Bbb{Z}\}.
$$
Let $S_N$ be the set of inequivalent cusps of $\G_0(N)$ and $S_N^*=S_N\backslash\{\i\}$. 

If $g(\t)=\sum_{n=0}^\i a_g(n)q^n$ is a meromorphic modular form of weight 2 on $\G_0(N)$, then $g(\t)$ has a Fourier expansion at each cusp $\mathfrak s$ in the form
$$
(g|_2\gamma_{\mathfrak s})(\t)=\sum_{n\geq n_{0}^{(\mathfrak s)}}a_g^{(\mathfrak s)}(n)q^{n/h_{\mathfrak s}}.
$$
Here $|_{k}$ is the usual weight $k$ slash operator. Now we use the canonical quotient map $\pi$ from $\H\cup\Q\cup\{\i\}$ to $X_0(N)$ to consider  $\omega_{g}=g(\t)d\t$, a meromorphic 1-form on $X_0(N)$. 

Then we see that for each cusp $\mathfrak s$ and $\tau \in \H$,
\begin{equation}\label{res1}
\text{Res}_{\pi(\mathfrak s)}\omega_{g}=\frac{h_{\mathfrak s}}{2 \pi i }a_g^{(\mathfrak s)}(0)
\quad  \mathrm{and}\quad \hbox{Res}_{\pi(\tau)}\omega_{g}={e_{N,\tau}}\hbox{Res}_{\tau}g.
\end{equation}

We recall that  $\displaystyle{f_\theta:=\frac{\partial_k (f)}{f}=\frac{\Theta(f)}{f}-\frac{kE_2}{12}}$ is a meromorphic modular form of weight $2$ on $\G_0(N)$. Moreover, it is holomorphic at each cusp and has only simple poles, if any, with the residue at each point $\t\in X_0(N)$ 
\begin{equation}\label{res2}
\hbox{Res}_{\pi(\tau)}\omega_{f_\theta}=\frac{1}{2\pi i}{e_{N,\t}}\hbox{ord}_{\tau}f.
\end{equation}
 Hence we have from \eqref{res1} and \eqref{res2} that for any $\t\in \F_N$, 
\begin{equation}\label{res3}
{e_{N,\t}}\hbox{ord}_{\tau}f={2\pi i}({e_{N,\t}})\hbox{Res}_{\tau}(f_\theta)
\end{equation}
and for each cusp $\mathfrak s$,
\begin{equation}\label{cons}
a_{f_\theta}^{(\mathfrak s)}(0)=\frac{\hbox{ord}_{\mathfrak s}f}{h_{\mathfrak s}}-\frac{k}{12}.
\end{equation}

Let $P_g$ be the set of singular points of $g(\t)$ on $\F_N$. In addition,
we let $\F_N(g,\ve)$ denote a punctured fundamental domain for $\G_0(N)$, which is obtained from $\F_N$ by deleting sufficiently small $\ve$-neighborhoods of singular points of $g$ and inequivalent cusps of $\G_0(N)$.  We further let $\g(t,\ve)$ denote the circle around $t$ with the radius $\ve$ so that $\partial\F_N(g,\ve)=\bigcup_{t\in P_g\cup S_N}\g(t,\ve)$. Following arguments in \cite[Proposition 3.5]{BF}, we use 
$$d(g\cdot \f_{N,n}d\t)=g\cdot\overline{\xi_0(\f_{N,n}(\t))}dxdy,$$
where $\t=x+iy$ and $\xi_0$ is the differential operator $\xi_0=2i \overline{\frac{\partial}{\partial\bar \t}}$ and apply Stokes' theorem to have
\begin{align}\label{stoke}
&\int_{\F_N(g,\ve)}g(\t)\cdot\overline{\xi_0(\f_{N,n}(\t))}dxdy\cr
&\qquad= \sum_{\frak{s}\in S_N}\int_{\g(\frak s,\ve)}g(\t)\f_{N,n}(\t)d\t+\sum_{z\in P_g}\int_{\g(z,\ve)}g(\t)\f_{N,n}(\t)d\t.
\end{align}

Since $\xi_0(\f_{N,m}(\t))=0$, using the same computation as \cite[Lemma 3.1]{Choi}, we find  that
\begin{align}\label{c1}
0=&\lim_{\ve\to 0}\int_{\F_N(g,\ve)}g(\t)\cdot\overline{\xi_0(\f_{N,n}(\t))}dxdy\cr
&\quad = a_g(n)+\sum_{\ell=1}^g a_N(n,-\ell)a_g(\ell)+\sum_{\frak{s}\in S^*_N}h_\frak{s} g(\frak{s})\f_{N,n}(\frak{s})+\sum_{z\in P_g}{2\pi i}({e_{N,z}})Res_z(g)\f_{N,n}(z).
\end{align}

Now we consider $\mathcal E_2(N)$, the Eisenstein space of weight 2 on $\G_0(N)$. We can always choose ${E}_{f}(\t)\in \mathcal E_2(N)$ so that its constant coefficient is identical to that of $f_\theta$ at each cusp of  $\G_0(N)$ except for the infinity cusp: Let $S_N=\{{\mathfrak s}_1,{\mathfrak s}_2,\dots {\mathfrak s}_t=\i\}$ be the set of inequivalent cusps of $\G_0(N)$. Then $\mathrm{dim}\ \mathcal E_2(N)=t-1$. Suppose $\mathcal E_2(N)$  is spanned by $\{E^{(1)},E^{(2)},\dots,E^{(t-1)}\}$. Recall from \eqref{cons} that the constant term of $f_\theta$ at each cusp ${\mathfrak s}_i$ is given by
\begin{equation}\label{fthetacons}
c_i:=\frac{\hbox{ord}_{{\mathfrak s}_i}f}{h_{{\mathfrak s}_i}}-\frac{k}{12}.
\end{equation}
By solving the system of equations 
\begin{equation}\label{se}\a_1E^{(1)}({\mathfrak s}_i)+\a_2E^{(2)}({\mathfrak s}_i)+\cdots+\a_{t-1}E^{(t-1)}({\mathfrak s}_i)=c_i\quad (1\leq i<t)
\end{equation}
for $(\a_1,\a_2,\dots,\a_{t-1})$, we obtain the desired modular form
\begin{equation}\label{se1}{E}_{f}(\t)=\a_1E^{(1)}(\t)+\a_2E^{(2)}(\t)+\cdots+\a_{t-1}E^{(t-1)}(\t).
\end{equation}
Since the coefficient matrix of the system of equations in \eqref{se} is non-singular, which can be easily shown by the Residue Theorem, 
the existence and the uniqueness of such modular form are guaranteed. 

We shall take $g=f_\theta-E_f$ in \eqref{c1}. Then by \eqref{serco} and \eqref{res3}, we obtain \eqref{divf}. Multiplying both sides of \eqref{divf} by $q^n$ and summing on $n$, and then applying the coefficient duality in \eqref{dual} yields \eqref{gtfne}.

%

\section{Proof of Theorem \ref{eis}}

Define for any point $(a,b)\in(\Z/N\Z)^2$ of order $N$, 
$$G_2^{(a,b)}(\tau):=\sum_{\substack{n=1\\ \gcd(n,N)=1}}^\i \sum_{\substack{(c,d)\equiv (a,b) (N)\\ \gcd(c,d)=n}} \frac{1}{(c\t+d)^2}.$$
For a primitive Dirichlet character $\varphi$ modulo $u$ such that $u^2=N$, the Eisenstein series $G_2^{\varphi,\bar\varphi}(\tau)$ is defined by
\begin{equation}\label{eisg}
G_2^{\varphi,\bar\varphi}(\tau):=\sum_{c=0}^{u-1}\sum_{d=0}^{u-1}\sum_{e=0}^{u-1}\varphi(cd)G_2^{(cu,d+eu)}(\tau)
\end{equation}
and the normalized Eisenstein series $E_2^{\varphi,\bar\varphi}(\tau)$ is given by
\begin{equation}\label{neis}
E_2^{\varphi,\bar\varphi}(\tau):=-\frac{u^2}{4\pi^2g(\varphi)}\frac{}{}G_2^{\varphi,\bar\varphi}(\tau),
\end{equation}
where $g(\varphi):=\displaystyle\sum_{n=0}^{u-1}\varphi(n)e^{\frac{2\pi n i}{u}}$ is the Gauss sum.

It follows from \cite[Theorem 4.6.2]{DS} that $\{E_2(\t)-dE_2(d\t):d|N, d\neq 1 \}\cup\{E_2^{\varphi,\bar\varphi}(t\tau): 1<tu^2|N, \varphi\ {\mathrm{is\ not\ trivial}} \}$ is a basis of $\mathcal E_2(N)$.
\par
Hence if $N$ is square-free, $\{E_2(\t)-dE_2(d\t): d|N, d\neq 1\}$ is a basis of $\mathcal E_2(N)$ and 
$$E_{f}(\t)=\sum_{d|N, d\neq 1} \a_d(E_2(\t)-dE_2(d\t))$$
for some constants $\a_d$.
In \cite[Lemma 5.1 (2)]{CLL}, it is proved that the constant term of $E_2(\tau)-dE_2(d\tau)$ at cusp $\frac{1}{v}$ is $1-\frac{\gcd(d,v)^2}{d}$ and thus, due to \eqref{fthetacons}, \eqref{se} and \eqref{se1}, the Fourier coefficients of $E_f$ are algebraic.

Now assume $N$ is not square-free. Then we have to consider both $E_2(\t)-dE_2(d\t)$ and $E_2^{\varphi,\bar\varphi}(t\tau)$. We first note that it holds that
\begin{equation}\label{egt}
E_2|_2\gamma(\tau)=E_2(\tau)-\dfrac{6ci}{\pi(c\tau+d)}\ 
\mathrm{for\ any}\ \gamma=\begin{pmatrix}a&b\\c& d\end{pmatrix}\in\Gamma(1),
\end{equation}
and then compute the constant coefficients of all $E_2(\t)-dE_2(d\t)$ appearing in the basis for $\mathcal E_2(N)$.

\begin{prop}\label{ede} Suppose $N$ is not square-free and $v$ is a divisor of $N$ such that $\gcd(v,N/v)>1$.
Let $\frac{e}{v}$ be a cusp on $X_0(N)$ with $e>1$.
Then the constant term of $E_2(\tau)-dE_2(d\tau)$ at $\frac{e}{v}$ is equal to $1-\frac{\gcd(d,v)^2}{d}$, which is the constant term of $E_2(\tau)-dE_2(d\tau)$ at $\frac{1}{v}$. 
\end{prop}
\begin{proof}
We need to compute $(E_2(\tau)-dE_2(d\tau))|_2\gamma$, where $\gamma \in \G(1)$ satisfies $\gamma( \infty) =\frac{e}{v}$. 
Taking $f$ and $h$ so that $eh-fv=1$, we obtain such $\gamma=\begin{pmatrix}e&f\\v&h\end{pmatrix}$. 
Let $s=\gcd(d,v)$ and let $d=d_1 s$ and $v=v_1 s$.
Since $\gcd(v_1,d_1e)=1$, we can choose $a,b\in\Z$ so that $ad_1e-bv_1=1.$ If we let $\alpha=\begin{pmatrix}d_1e&b\\v_1&a\end{pmatrix}\in\Gamma(1)$, then
\begin{equation}\label{d001}
\begin{pmatrix}d&0\\0&1\end{pmatrix}\begin{pmatrix}e&f\\v&h\end{pmatrix}=\alpha\begin{pmatrix}s&k\\0& d_1\end{pmatrix},
\end{equation}
where $k=adf-bh$. Let $\beta=\begin{pmatrix}s&k\\0& d_1\end{pmatrix}$.
%
By \eqref{egt}, we have
\begin{equation}\label{e1}E_2|_2\gamma(\tau)=E_2(\tau)-\dfrac{6vi}{\pi(v\tau+h)}.
\end{equation}
Applying \eqref{d001} in the second and \eqref{egt} in the last equality below, we have
\begin{equation}\label{e2}dE_2(d\tau)|_2\gamma=E_2|_2\sm d&0\\0&1\esm\gamma(\tau)=E_2|_2\alpha\beta(\tau)=\left(E_2(\tau)-\dfrac{6v_1i}{\pi(v_1\tau+a)}\right)|_2\beta.
\end{equation}
Here,
\begin{equation}\label{e3}E_2|_2\beta(\tau)=dE_2(\beta(\tau))d_1^{-2}=\frac{\gcd(d,v)^2}{d}E_2\left(\frac{\gcd(d,v)^2}{d}\tau+\frac{k\gcd(d,v)}{d}\right)\end{equation}
and
\begin{equation}\label{e4}\dfrac{6v_1i}{\pi(v_1\tau+a)}|_2\beta=d\dfrac{6v_1i}{\pi(v_1\beta(\tau)+a)}d_1^{-2}=\dfrac{6vi}{\pi(v\tau+h)}.
\end{equation}
It follows from \eqref{e1} through \eqref{e4} that 
$$(E_2(\tau)-dE_2(d\tau))|_2\gamma=E_2(\tau)-\frac{\gcd(d,v)^2}{d}E_2\left(\frac{\gcd(d,v)^2}{d}\tau+\frac{k\gcd(d,v)}{d}\right).$$
Hence the constant term of $E_2(\tau)-dE_2(d\tau)$ at $\frac{e}{v}$ is $\displaystyle{1-\frac{\gcd(d,v)^2}{d}}$.
\end{proof}
Now we compute the constant coefficients of $E_2^{\varphi,\bar\varphi}(d\t)$ appearing in the basis for $\mathcal E_2(N)$.

\begin{prop}\label{ct} Suppose $N=u^2$ for some positive integer $u$. Let $\varphi$ be a primitive Dirichlet character modulo $u$ and let $\frac{e}{v}$ with $v|N$ be any cusp of $X_0(N)$. Then the constant term $c_{\frac{e}{v}}$ of $E_2^{\varphi,\bar\varphi}(\tau)$ at $\frac{e}{v}$ is given by
\begin{align*}
c_{\frac{e}{v}}=\begin{cases}\displaystyle-\frac{1}{2g(\varphi)}\sum_{d\in(\Z/u\Z)^*}\varphi(-ed^2)\frac{1}{1-\cos(2d\pi/u)},& v=u, \\
0,&\hbox{otherwise}.\end{cases}
\end{align*}
\end{prop}

\begin{proof}
We first compute the constant coefficient of $G_2^{\varphi,\bar\varphi}(\tau)$.
Note that $\Gamma(1)=\bigcup_{(r,s)}\Gamma_0(N)\gamma_{(r,s)}$
where $\{\gamma_{(r,s)}=\sm p&q\\r&s\esm \,|\,\gcd(r,s)=1,s|N,0<r\leq \frac{N}{s}\}$ is a set of coset representatives of $\Gamma_0(N)\backslash\Gamma(1)$. For computing the constant term of $G_2^{\varphi,\bar\varphi}|_2\gamma(\tau)$ for $\gamma\in\G(1)$, it suffices to compute $G_2^{\varphi,\bar\varphi}|_2\gamma_{(r,s)}(\tau)$ for some $(r,s)$, because $\Gamma_0(N)$ acts on $G_2^{\varphi,\bar\varphi}$ trivially (See \cite[p.127]{DS}.) By \eqref{eisg}, we see that
$$G_2^{\varphi,\bar\varphi}|_2\gamma(\tau)=\sum_{c=0}^{u-1}\sum_{d=0}^{u-1}\sum_{e=0}^{u-1}\varphi(cd)G_2^{(cu,d+eu)\gamma}(\tau).$$

Let us consider the case of $s>1$. Since $(cu,d+eu)\gamma_{(r,s)}=(dr+(cp+er)u,ds+(cq+es)u)$, the constant term of $G_2^{(cu,d+eu)\gamma_{(r,s)}}(\tau)$ is equal to $0$ unless $dr+(cp+er)u\equiv 0\pmod{N}$ by \cite[Eq.(4.23)]{DS}. Since $(cu,d+eu)\in(\Z/N\Z)^2$ is of order $N$, so is $(cu,d+eu)\gamma$.   But if $dr+(cp+er)u\equiv 0\pmod{N}$, then $(cu,d+eu)\gamma$ cannot be of order $N$, because $\gcd(ds+(cq+es)u,N)>1$.
Thus the constant term of $G_2^{(cu,d+eu)\gamma_{(r,s)}}(\tau)$ is equal to $0$ and hence so is the constant term of $G_2^{\varphi,\bar\varphi}|_2\gamma_{(r,s)}(\tau)$.

The next case to consider is when $s=1$ and $u\nmid r$. In this case, we may write
$\gamma=\sm 1&0\\r&1\esm$ for $r=1,2,\dots,N$.
For these $\g$,  $(cu,d+eu)\gamma=(dr+(c+er)u,d+eu)$.
If $dr+(c+er)u\equiv 0\pmod{N}$, then 
$u|\gcd((c+er)u,N)$ and $\gcd((c+er)u,N)|dr$, which implies $u|dr$.
But this cannot occur, because $\gcd(d,N)=1$ and $u\nmid r$. Hence, again by \cite[Eq.(4.23)]{DS}, the constant term of $G_2^{(cu,d+eu)\gamma}(\tau)$ is equal to $0$, and thus so is the constant term of $G_2^{\varphi,\bar\varphi}|_2\gamma(\tau)$.

Finally, we assume that $s=1$ and $u|r$, that is,  $\gamma=\sm 1&0\\ \ell u&1\esm$ for $\ell=1,2,\dots,u$. It follows from $(cu,d+eu)\g =(cu+d\ell u,d+eu)$ and \cite[Eq.(4.23)]{DS} that a non-zero constant term of $G_2^{\varphi,\bar\varphi}|_2\gamma(\tau)$ may occur only when $cu+d\ell u\equiv 0\pmod{N}$. 
Since $\gcd(c,N)=1$, $c+d\ell \equiv 0\pmod{u}$ implies that $\gcd(\ell,N)=1$, and hence $c=-d\ell$.
Thus for each $d,\ell\in(\Z/u\Z)^*$, we have
$$G_2^{\varphi,\bar\varphi}|_2\sm 1&0\\ \ell u&1\esm (\tau)=\sum_{e=0}^{u-1}\sum_{d\in(\Z/u\Z)^*}\varphi(-\ell d^2)G_2^{(0,d+eu)}(\tau).$$
By \cite[Eq.(4.23)]{DS}, the constant term of $G_2^{\varphi,\bar\varphi}|_2\sm 1&0\\ \ell u&1\esm (\tau)$ is then
\begin{equation*}
\sum_{e=0}^{u-1}\sum_{d\in(\Z/u\Z)^*}\varphi(-\ell d^2)\zeta_N^{d+eu}(2)
=\sum_{d\in(\Z/u\Z)^*}\varphi(-\ell d^2)\zeta_u^{d}(2),
\end{equation*}
where $\zeta_t^{a}(2)=\displaystyle\sum_{m\equiv a\pmod{t}}\frac{1}{m^2}$ for $a$ not congruent to $0$ modulo $t$.
Hence using \eqref{neis} and \cite[Lemma 3.6]{KK} in turn, we find that the constant term of $E_2^{\varphi,\bar\varphi}|_2\sm 1&0\\ \ell u&1\esm (\tau)$ is 
\begin{equation}\label{cte}
-\frac{u^2}{4\pi^2g(\varphi)}\sum_{d\in(\Z/u\Z)^*}\varphi(-\ell d^2)\zeta_u^{d}(2)=-\frac{1}{2g(\varphi)}\sum_{d\in(\Z/u\Z)^*}\varphi(-\ell d^2)\frac{1}{1-\cos(2d\pi/u)}.
\end{equation}

Now we are ready to compute the constant terms of $E_2^{\varphi,\bar\varphi}(\tau)$ at cusps $\frac{e}{v}$ with $v|N$ and $e>1$.
First suppose all prime factors of $N$ divide $v$.
For $\gamma=\sm e&f\\v&h\esm\in\G(1)$ such that $\g(\i)=\frac{e}{v}$, the congruence $hx\equiv v\pmod{N}$ has a unique solution $x=r$, because $\gcd(h,v)=1$ and $\gcd(h,N)=1$.
Then 
\begin{align*}
\begin{pmatrix}e&f\\v&h\end{pmatrix}=\begin{pmatrix}e-fr&f\\ v-hr&h\end{pmatrix}\begin{pmatrix}1&0\\r&1\end{pmatrix}.
\end{align*}
Thus $\gamma=\delta\sm 1&0\\ r&1\esm$ for some $\delta\in\Gamma_0(N)$.  Here $v|r$.
Moreover, we note that $u|v$ if and only if $u|r$.
Hence by the arguments above, the constant term of $E_2^{\varphi,\bar\varphi}|_2\gamma(\tau)$ may not be zero only when $r=\ell u$ or equivalently  $v=t u$ for some $t |u$. But since $\gcd(\ell,u)=1$, $t=1$. 
If we take $h^{-1}$ modulo $N$, then $h^{-1}\equiv e\pmod{v}$, because $eh-vf=1$ and $v|N$. Hence  $r\equiv h^{-1}v\equiv h^{-1}\ell u\equiv (e+n\ell u)\ell u\pmod{N}$ for some $n$. That is, $r\equiv ev\pmod{N}$.
Consequently, we have a non-zero constant term of $E_2^{\varphi,\bar\varphi}(\tau)$ at the cusp $\frac{e}{v}$ only when $v=u$, which is the constant term of $E_2^{\varphi,\bar\varphi}|_2{\sm 1&0\\ eu&1\esm}(\tau)$ that is computed in \eqref{cte}. 

Next, we suppose a prime factor $p$ of $N$ does not divide $v$.
Since $\gcd(ep,v)=1$, there exist $f,h$ so that $eph-fv=1$.
Then $\gamma=\sm e&f\\v&ph\esm\in\G(1)$ satisfies $\g(\i)=\frac{e}{v}$.
But $\gamma$ is not of the form $\delta\sm 1&0\\ r&1\esm$ for any $\delta\in\Gamma_0(N)$, because as $p|\gcd(ph,N)$, $\gcd(ph,N)\nmid v$, which means the congruence $phx\equiv v\pmod{N}$ has no solution.
Thus $\gamma=\delta\gamma_{(r,s)}$ for some $\delta\in\Gamma_0(N)$ and $s>1$, which implies that the constant term of $E_2^{\varphi,\bar\varphi}|_2\gamma(\tau)$ is equal to $0$.
\end{proof}

Finally, we compute the constant term $c_{\frac{e}{v}}$ of $E_2^{\varphi,\bar\varphi}(t\tau)$ at the cusp $\frac{e}{v}$ under the condition $1<tu^2|N$ and $v|N$.

\begin{lem}\label{cusp1} Let $\mathfrak s=\frac{e}{v}$ be a cusp on $X_0(N)$ and $M|N$.
Then the cusp $\mathfrak s$ is $\Gamma_0(M)$-equivalent to $\mathfrak s'=\frac{me}{v'}$ with $v'=\gcd(v,M)$ and $m=\frac{v}{v'}$.
Moreover, if $f$ is a modular form of weight $k$ on $\Gamma_0(M)$,
then the constant term of $f(\tau)$ at $\mathfrak s$ is the same as the constant term of $f(\tau)$ at $\mathfrak s'$.
\end{lem}
\begin{proof} Since $\gcd(Me,v)=v'$, we can choose relatively prime integers $c$ and $d$ such that 
\begin{equation}\label{vv'}
cMe+dv=v',
\end{equation} 
from which we have $ce\frac{M}{v'}+dm=1$. Thus if we let $g=\gcd(v',\frac{M}{v'})$, then we have $dm\equiv 1 \pmod g$.
On the other hand, applying \cite[Theorem A]{L} in \eqref{vv'}, we may assume $\gcd(d,Me)=1$.
Then $\gcd(cM,d)=1$, and we can choose $a$ and $b$ so that $ad-cbM=1$.
Since $a\equiv d^{-1}\equiv m \pmod g$, we find that $\mathfrak s'=\delta \mathfrak s$ where $\delta:=\sm a&b\\cM&d\esm\in\Gamma_0(M)$. 

 Assume that $\gamma'(\infty)=\mathfrak s'$ for $\g'\in\G(1)$. Since  $\mathfrak s'=\delta \mathfrak s$, $\delta^{-1}\gamma'(\infty)=\mathfrak s$, and hence the constant term of $f|_k\delta^{-1}\gamma'$ is the same as the constant term of $f|_k\gamma'$.
\end{proof}

Note that $E_2^{\varphi,\bar\varphi}(t\tau)\in \mathcal E_2(M)$ when $M=tu^2|N$. Hence by Lemma \ref{cusp1}, it suffices to compute the constant term of $E_2^{\varphi,\bar\varphi}(t\tau)$ at the cusps of $\G_0(M)$ in order to obtain  the constant term at the cusps of $\G_0(N)$.

\begin{prop}\label{ctt} Let $N=tu^2$ for some positive integer $u$. Let $\varphi$ be a primitive Dirichlet character modulo $u$ and let $\frac{e}{v}$ with $v|N$ be any cusp of $X_0(N)$. Then the constant term $c_{\frac{e}{v}}$ of $E_2^{\varphi,\bar\varphi}(t\tau)$ at $\frac{e}{v}$ is given by
\begin{align*}
c_{\frac{e}{v}}=\begin{cases}\displaystyle
-\frac{1}{2tg(\varphi)}\sum_{d\in(\Z/u\Z)^*}\varphi(-etd^2/\gcd({v},t))\frac{1}{1-\cos(2d\pi/u)},& v=u\gcd(v,t),\\
0,&\hbox{otherwise}.\end{cases}
\end{align*}
\end{prop}

\begin{proof}
We choose $e,f$ so that $eh-fv=1$. Then $\gamma=\begin{pmatrix}e&f\\v&h\end{pmatrix}\in \Gamma(1)$ and $\g(\i)=\frac{e}{v}.$
Let $s=\gcd(t,v)$ and write $t=t_1 s$ and $v=v_1 s$.  As in the proof of Proposition \ref{ede}, we have 
\begin{equation*}
\begin{pmatrix}t&0\\0&1\end{pmatrix}\gamma=\alpha\beta,
\end{equation*}
where $\alpha=\begin{pmatrix}t_1e&b\\v_1& a\end{pmatrix}\in\Gamma(1)$ and $\beta=\begin{pmatrix}s&k\\0& t_1\end{pmatrix}$ with $k=atf-bh$.
Hence
$$E_2^{\varphi,\bar\varphi}(t\tau)|_2\gamma=\frac{1}{t}E_2^{\varphi,\bar\varphi}|_2\sm t&0\\0&1\esm\gamma(\tau)=\frac{1}{t}E_2^{\varphi,\bar\varphi}|_2\alpha\beta(\tau)$$
and the constant term we look for is that of $\frac{1}{t}E_2^{\varphi,\bar\varphi}(\tau)$ at the cusp $\sm et_1\\v_1\esm$.
By Proposition \ref{ct}, it is non-zero if and only if $v_1=u$, i.e. $v=us$. 
\end{proof}

We have computed the constant coefficients of basis elements of $\mathcal E_2(N)$ at each cusp in Propositions \ref{ede}, \ref{ct} and \ref{ctt} when $N$ is not square-free. Since they are all algebraic, the Fourier coefficients of $E_f(\t)$ should be algebraic by \eqref{fthetacons}, \eqref{se} and \eqref{se1}. 

\section{Proof of Theorem \ref{alg}}

First, we briefly describe how to construct basis elements $\mathfrak f_{N,m}$ of $M_0^\sharp(N)$ with algebraic Fourier coefficients  using the method of finding two generators of the function field of $\Gamma_0(N)$ introduced by Yang.
Following Yang \cite{Yang}, for $a$ not congruent to 0 modulo $N$, we define a generalized Dedekind $\eta$-function by
$$\eta_a(\tau)=q^{NB(a/N)/2}\prod_{m=1}^\infty\left(1-q^{(m-1)N+a}\right)\left(1-q^{mN-a}\right),$$
where $B(x)=x^2-x+\frac{1}{6}$.
Consider the function 
\begin{align}\label{ftau}
f(\tau)=\prod_a \eta_a(\tau)^{r_a},
\end{align}
where $a$ and $r_a$ are integers with $a\not\equiv0\pmod{N}$.
By \cite[Proposition 3]{Yang}, $f$ is a modular function of $\Gamma_1(N)$ if the following hold:
\begin{enumerate}
\item[(i)] $\sum_a r_a \equiv 0 \pmod{12}$
\item[(ii)] $\sum_a a r_a \equiv 0 \pmod{2}$
\item[(iii)] $\sum_a a^2 r_a \equiv 0 \pmod{2N}$
\end{enumerate}
We note that $f$ satisfying the conditions (i),(ii), and (iii) has integer Fourier coefficients, because $\eta_a(\tau)$ does so.
Yang proved that there are at least two functions $X$ and $Y$ of $\Gamma_1(N)$ that have poles only at $\infty$ and whose orders of poles are relatively prime.
Then $X$ and $Y$ are generators of the function field of $\Gamma_1(N)$.

To construct a modular function of $\Gamma_0(N)$ with a pole of order $m$ at $\infty$ and analytic elsewhere, Yang first found a function $f$ on $\Gamma_1(N)$ that has a pole of order $m$ at $\infty$, poles of order $< m$ at other cusps equivalent to $\infty$ under $\Gamma_0(N)$, and regular at any other points. 
Then the function
$$\sum_{\gamma\in\Gamma_0(N )/\Gamma_1(N)}f (\gamma\tau)$$
is modular of $\Gamma_0(N)$ with the desired properties.
In general, this argument holds for any intermediate subgroup $\Gamma$ between $\Gamma_1(N)$ and $\Gamma_0(N)$, which is normal in $\Gamma_0(N)$.
We note that if $f$ of \eqref{ftau} is a modular function of $\Gamma$, then $\sum_{\gamma\in\Gamma_0(N )/\Gamma_1(N)}f (\gamma\tau)$ has algebraic (most likely rational) Fourier coefficients because of the transformation formula in \cite[Proposition 2]{Yang}.

For example, we construct a basis $\{\mathfrak f_{31,m}\}$ of $M_0^\sharp(31)$.
If we find modular functions of $\Gamma_0(31)$ with unique poles of order $3,4$ and $5$ at $\infty$, 
then by using them we can construct $\mathfrak f_{31,m}$ for all $m\geq 3$.
Let $\Gamma$ be the subgroup generated by $\Gamma_1(31)$ and $\begin{pmatrix}5& -1\\31& -6\end{pmatrix}$.
Then $f_k=\frac{\eta_{6k}\eta_{26k}\eta_{30k}}{\eta_{2k}\eta_{10k}
\eta_{12k}}$ is a modular function of $\Gamma$ for any integer $k$ not divisible by $31$.
There are essentially five distinct $f_k$, and they are $f_1$, $f_2$, $f_3$, $f_4$, and $f_8$. 
Moreover, the cusp $\infty$ splits into five cusps $\frac{1}{31}$, $\frac{2}{31}$, $\frac{3}{31}$, $\frac{4}{31}$, and $\frac{8}{31}$ in $\Gamma$.
The orders of $f_k$ at those cusps are as follows:

\begin{center}
\begin{longtable}{c|ccccc}
\label{word} & $\frac{1}{31}$ & $\frac{2}{31}$ & $\frac{3}{31}$ & $\frac{4}{31}$ & $\frac{8}{31}$
 \\ \hline
$f_1$ & $3$ & $0$ & $-4$ & $2$ & $-1$
 \\ 
$f_2$ & $0$ & $2$ & $3$ & $-1$ & $-4$
 \\ 
$f_3$ & $-4$ & $3$ & $-1$ & $0$ & $2$
 \\ 
$f_4$ & $2$ & $-1$ & $0$ & $-4$ & $3$
 \\ 
$f_8$ & $-1$ & $-4$ & $2$ & $3$ & $0$
\end{longtable}
\end{center}

We solve the integer programming problem
$$\systeme{
3x_1 - 4x_3 + 2x_4 - x_5 = -3,
2x_2 + 3x_3 - x_4 - 4x_5 \geq -3,
-4x_1 + 3x_2 - x_3 + 2x_5 \geq -3,
2x_1 - x_2 - 4x_4 + 3x_5 \geq -3,
-x_1 - 4x_2 + 2x_3 + 3x_4 \geq -3,}$$
to find a solution $(x_1,x_2,x_3,x_4,x_5)=(0,0,1,1,1)$
so that we obtain
$$\sum_{\gamma\in\Gamma_0(N )/\Gamma}f_3 f_4 f_8 (\gamma\tau)=\frac{1}{q^3} + \frac{2}{q^2} + 2 - q + 3q^2 + 2q^3 + q^4 + 2q^5-\cdots,$$
which is invariant under $\Gamma_0(31)$ and has a unique pole of order $3$ at $\infty$.
Similarly, we have that 
\begin{align*}
\sum_{\gamma\in\Gamma_0(N )/\Gamma}f_3(\gamma\tau)&=\frac{1}{q^4}+\frac{1}{q^3}+\frac{1}{q^2}+\frac{1}{q}-1+q+3 q^2+q^3+q^4+\cdots \hbox{ \, and }\\
\sum_{\gamma\in\Gamma_0(N )/\Gamma}f_3 f_8(\gamma\tau)&=\frac{1}{q^5}+\frac{1}{q^4}+\frac{1}{q^3}+\frac{1}{q^2}+2+q+5 q^2+2q^3-q^4+2q^5+\cdots,
\end{align*}
which are invariant under $\Gamma_0(31)$ and have unique poles of order $4$ and $5$ at $\infty$, respectively.
By using them, we have the following basis elements with rational coefficients:
\begin{align*}
\mathfrak f_{31,3}&=\frac{1}{q^3} + \frac{2}{q^2} - q + 3q^2 + 2q^3 + q^4 + 2q^5-\cdots,\\
\mathfrak f_{31,4}&=\frac{1}{q^4}-\frac{1}{q^2}+\frac{1}{q}+2q-q^3-2q^5+\cdots,\\
\mathfrak f_{31,5}&=\frac{1}{q^5}-\frac{1}{q}+2q^2+q^3-2q^4+2q^5+\cdots\\
&\vdots
\end{align*}

Next, we establish a recurrence relation between Fourier coefficients of a meromorphic modular form and exponents in its product representation. 
\begin{prop}\label{rec}
Suppose
\begin{equation}\label{f1}
f(\t)=1+\sum_{m=1}^\infty a(m) q^m = \prod_{n=1}^\infty (1-q^n)^{c(n)}
\end{equation} is a non-zero meromorphic function in a neighborhood of $q=0$, then for each $m\geq 1$, it satisfies that
\begin{equation} \label{recursion}
c(m)= - a(m) -\frac 1m \left(\sum_{1\le u <m \atop u|m} u c(u) + \sum_{1\le k < m} a(m-k) \sum_{u|k} u c(u) \right).
\end{equation}
\end{prop}
\begin{proof}
Following the method in 
\cite[Section 4]{Kim} 
we observe that
\begin{align*}
\prod_{n=1}^\infty (1-q^n)^{c(n)} & = \prod_{n=1}^\infty\exp (\log (1-q^n)^{c(n)} ) = \exp (\sum_{n=1}^\infty c(n) \log (1-q^n))  
= \exp (-\sum_{n=1}^\infty c(n) \sum_{m=1}^\infty (q^n)^m/m) \\
& = \exp (-\sum_{m=1}^\infty \sum_{n=1}^\infty  n c(n)q^{nm}/(nm))
= \exp (-\sum_{m=1}^\infty \sum_{u|m}  u c(u)q^{m}/m) =:V.
\end{align*}
We note that $\Theta(\log V) = \Theta(V)/V$. 
Since $$V= 1+\sum_{m=1}^\infty a(m) q^m =  \exp (- \sum_{m=1}^\infty \sum_{u|m}  u c(u)q^{m}/m), $$ the relation $V\Theta(\log V) = \Theta(V)$ gives rise to 
$$
( \sum_{m=1}^\infty \sum_{u|m}  u c(u)q^{m})(1+\sum_{m=1}^\infty a(m) q^m )= - \sum_{m=1}^\infty m a(m) q^m. 
$$
Comparing the coefficients of $q^m$ on both sides, we get 
$$
\sum_{u|m} u c(u) + \sum_{1\le k < m} a(m-k) \sum_{u|k} u c(u) = - m a(m),
$$ 
which renders the desired recursion.
\end{proof}

Now we begin to prove Theorem \ref{alg}.
It follows from \eqref{divf}, Theorem \ref{eis},  Proposition \ref{rec} and algebraicity of coefficients of $\f_{N,n}$} that
\begin{equation}\label{alg11}
\sum_{z\in\F_N} e_{N,z} ord_z(f)\f_{N,n}(z)
\ \mathrm{is \ algebraic}.
\end{equation}

Before all else, we consider the case when the genus is $1$. It is well known that the values of $\f_{N,2}(z)$ and $\f_{N,3}(z)$ at elliptic points $z$ are algebraic by Shimura Reciprocity Theorem \cite[Theorem 15.12]{Cox} and there are only a finite number of elliptic points of $X_0(N)$.  
Hence if $z_1, z_2,\dots z_\ell$ are all of non-elliptic points of $X_0(N)$ for which $ord_z(f)\neq 0$ (not necessarily distinct), we find from \eqref{alg11}  that for every $n\geq 2$, 
\begin{equation}\label{alg12}\sum_{s=1}^\ell \f_{N,n}(z_s)=\sum_{s=1}^\ell Q_n(\f_{N,2}(z_s),\f_{N,3}(z_s) )\ \mathrm{is \ algebraic},\end{equation}
where $Q_n(X,Y)$ is a monic polynomial in $X$ and $Y$ with algebraic coefficients.
Using induction on $n+m$ ($n,m\geq 0$) with \eqref{alg12}, one can show that
$$\sum_{s=1}^\ell \f_{N,2}^n(z_s)\f_{N,3}^m(z_s)\ \mathrm{are \ algebraic}$$
for all $m\geq 0$ and $n\geq 0$. %
In particular, for every positive integer $n$,
$\displaystyle{\sum_{s=1}^\ell \f_{N,2}^n(z_s)}$ and  $\displaystyle{\sum_{s=1}^\ell \f_{N,3}^n(z_s)}$ are algebraic.
Thus by solving for the elementary symmetric functions in $\f_{N,i}(z_1)$,  $\f_{N,i}(z_2)$,\dots,  $\f_{N,i}(z_\ell)$ for $i=1,2$, we find that
$\displaystyle{\prod_{s=1}^\ell (x-\f_{N,2}(z_s))}$ and $\displaystyle{\prod_{s=1}^\ell (x-\f_{N,3}(z_s))}$ are polynomials with algebraic coefficients. Therefore for every $z_0\in\F_N$ for which $\textrm{ord}_{z_0}(f)\neq 0$, $\f_{N,2}(z_0)$ and $\f_{N,3}(z_0)$ are both algebraic, and hence $\f_{N,n}(z_0)=Q_n(\f_{N,2}(z_0),\f_{N,3}(z_0) )$ are algebraic for all $m\geq2$.

For a general genus $g$, we use the fact
$$\f_{N,n}(z_s)=Q_n(\f_{N,g+1}(z_s),\f_{N,g+2}(z_s),\dots,\f_{N,2g+1}(z_s))$$ for some polynomial $Q_n(X_1,X_2,\cdots,X_{g+1})$ with algebraic coefficients and the same arguments above to show $\f_{N,n}(z_0)$ are algebraic for all $n\geq g+1$. 

\section{Examples}
\begin{exam}\label{11} Consider a weight $2$ modular form on $\G_0(11)$
$$f(\t):=- \frac{1}{10}(E_2(\t)-11E_2(11\t) +24\Delta_{11}(\t)),$$ where
$\Delta_{11}(\t):=\fg_{11,-1}(\t)$ is the unique normalized cusp form on $\G_0(11)$. This function is related with $Z_{K3}(z;\t)$, the elliptic genus of $K3$. More precisely, the twisted elliptic genus of $11A$ type conjugacy class of Mathieu group $M_{24}$ is given by \cite[Eq. (2.13)]{EH}
$$Z_{11A}(z;\t)=\frac{1}{12}Z_{K3}(z;\t)-\frac{11}{6}f(\t)\frac{\theta_{11}(z;\t)^2}{\eta(\t)^6},$$
where $\theta_{11}(z;\t)$ is the Jacobi theta function defined by $\theta_{11}(z;\t)=\sum_{n\in\Z}q^{\frac12(n+\frac12)^2}e^{2\pi i (n+\frac12)(z+\frac12)}$. Suppose $z_1$ and $z_2$ are zeros of $f(\t)$ in the fundamental domain $\mathcal F_{11}$, then the algebraic values $\f_{11,2}(z_1)$ and $\f_{11,2}(z_2)$ are roots of a quadratic polynomial $X^2+22X+233$.
\end{exam}
\begin{proof} 
By valence formula, $f$ has two zeros in $X_0(11)$. Moreover, applying the Residue Theorem to the 1-form $f(\t)d\t$, we find that 
\begin{align*}
& (\hbox{width of the cusp $\infty$}) \cdot (\hbox{constant term of $f(\t)$ at the cusp $\infty$}) \\
& +
(\hbox{width of the cusp $0$}) \cdot (\hbox{constant term of $f(\t)$ at the cusp $0$})
=0.
\end{align*} 
Hence we find that the constant term of $f(\t)$ at the cusp $0$ is  $- \frac{1}{11}$, and thus $f$ has two zeros, say $z_1$ and $z_2$ in $Y_0(11)$. 

On the other hand, if we write $$f(\t)=1+\sum_{m=1}^\infty a(m) q^m = \prod_{n=1}^\infty (1-q^n)^{c(n)},$$ 
then from the initial values $a(0)=1$, $a(1)=0$, $a(2)=12$, $a(3)=12$, $a(4)=12$, $\dots$, and Proposition  \ref{rec}, we have $c(1)=0$, $c(2)=-12$, $c(3)=-12$, $c(4)=66$, etc. 
Since $\mathcal E_2(11)$ is spanned by $E_2(\t)-11E_2(11\t)$ whose constant term at the cusp $0$ is $1-\frac{\gcd(11,1)^2}{11}=\frac{10}{11}$ and the constant term of $f_\theta$ at the cusp $0$ is $-\frac{1}{6}$ by \eqref{cons}, it is clear by \eqref{se} and \eqref{se1} that
$$E_f(\t)=-\frac{11}{60}(E_2(\t)-11E_2(11\t))=\sum_{n=0}^\i \ve_f(n)q^n=-\frac{1}{6}-\frac{2}{5}q-\frac{6}{5}q^2-\frac{8}{5}q^3-\frac{14}{5}q^4+\cdots.$$  
It then follows from \eqref{divf} with $n=2, 3, 4$ that 
$$
\f_{11,2}(z_1)+\f_{11,2}(z_2)=\sum_{d|2}c(d) d-12+\ve_{f}(2)+(c(1)-4+\ve_{f}(1) ) a_{11}(2,-1)=-22
$$
$$
\f_{11,3}(z_1)+\f_{11,3}(z_2)=\sum_{d|3}c(d) d-16+\ve_{f}(3)+(c(1)-4+\ve_{f}(1) ) a_{11}(3,-1)=-34,
$$
and
$$
\f_{11,4}(z_1)+\f_{11,4}(z_2)=\sum_{d|4}c(d) d-28+\ve_{f}(4)+(c(1)-4+\ve_{f}(1) ) a_{11}(4,-1)=242,
$$
because $a_{11}(2,-1)=2$, $a_{11}(3,-1)=1$, and $a_{11}(4,-1)=-2$.

Since $\f_{11,4}=\f_{11,2}^2 -4 \f_{11,3} -4 \f_{11,2}$, we have 
$$
x_1^{''}+ x_2^{''}=x_1^2+x_2^2- 4 (x_1^{'} + x_2^{'}) -4(x_1+x_2),
$$
where $x_i=\f_{11,2}(z_i)$, $x_i^{'}=\f_{11,3}(z_i)$ and $x_i^{''}=\f_{11,4}(z_i)$ for $i=1,2$.
We then obtain that $x_1^2+x_2^2=18$ and therefore $x_1 x_2=233$. 
Thus we conclude that $x_1=\f_{11,2}(z_1)$ and $x_2=\f_{11,2}(z_2)$ are roots of a quadratic polynomial $X^2+22X+233$. 
\end{proof}

\begin{exam}
For an example of a non square-free level, we take a look at the case when level $N=27$. We denote the cusps of $\G_0(27)$ except for the infinity by $\mathfrak s_1=0$, $\mathfrak s_2=\frac 13$, $\mathfrak s_3=\frac19$, $\mathfrak s_4=\frac23$, and $\mathfrak s_5=\frac29$.
We also let  $d_1=1$, $d_2=3$, $d_3=9$, and $d_4=27$.

Let $\varphi$ be a primitive Dirichlet character of conductor 3. Then 
$E^{(1)}(\t):=E_2(\tau)-3E_2(3\tau)$, $E^{(2)}(\t):=E_2(\tau)-9E_2(9\tau)$, $E^{(3)}(\t):=E_2(\tau)-27E_2(27\tau)$, $E^{(4)}(\t):=E_2^{\varphi,\bar\varphi}(\tau)$, and $E^{(5)}(\t):=E_2^{\varphi,\bar\varphi}(3\tau)$
form a basis for the Eisenstein space $\mathcal E_2(27)$.

According to Proposition \ref{ede}, the constant terms of $E^{(i)}(\t)$ at $\mathfrak s_j$ for $1\leq i\leq 3$ and $1\leq j\leq 3$ are
$$1-\frac{\gcd(d_{i+1},d_j)^2}{d_{i+1}}$$
and the constant terms of $E^{(i)}(\t)$ at $\mathfrak s_j$ for $1\leq i\leq 3$ and $j=4,5$ are
$$1-\frac{\gcd(d_{i+1},d_{j-2})^2}{d_{i+1}}.$$
Also, by Proposition \ref{ct} and Lemma \ref{cusp1}, the constant terms of $E^{(4)}(\t)$ at $\frak s_j$ for $1\leq j\leq 5$ are equal to $0$, $-\frac{2\sqrt{3}}{9}i$, $-\frac{2\sqrt{3}}{9}i$, $\frac{2\sqrt{3}}{9}i$ and $\frac{2\sqrt{3}}{9}i$, respectively. 
By Proposition \ref{ctt}, the constant terms of $E^{(5)}(\t)$ at $\frak s_3$ and $\frak s_5$ are equal to $-\frac{2\sqrt{3}}{27}i$ and $\frac{2\sqrt{3}}{27}i$, respectively, and $0$ at other cusps. 

Therefore, for a weight $k$ meromorphic modular form $f$ on $\G_0(27)$, solving the linear system in \eqref{se}, we have 
$$E_f(\t)=\a_1E^{(1)}(\t)+\a_2E^{(2)}(\t)+\a_3E^{(3)}(\t)+\a_4E^{(4)}(\t)+\a_5E^{(5)}(\t),$$
where
\begin{align*}
\a_1=&\frac{3}{8}c_1-\frac{5}{24}c_2+\frac{1}{48}c_3-\frac{5}{24}c_4+\frac{1}{48}c_5\\
\a_2=&-\frac{3}{8}c_1+\frac{1}{48}c_2-\frac{1}{12}c_3+\frac{1}{48}c_4-\frac{1}{12}c_5\\
\a_3=&\frac{9}{8}c_1+\frac{1}{8}c_1+\frac{1}{16}c_3+\frac{1}{8}c_4+\frac{1}{16}c_5\\
\a_4=&\frac{3\sqrt{3}i}{4}c_2-\frac{3\sqrt{3}i}{4}c_4\\
\a_5=&-\frac{9\sqrt{3}i}{4}c_2+\frac{9\sqrt{3}i}{4}c_3+\frac{9\sqrt{3}i}{4}c_4-\frac{9\sqrt{3}i}{4}c_5
\end{align*}
and $c_j$ is the constant term of $f_\theta$ at the cusp $\mathfrak s_j$ ($1\leq j\leq 5$).
\end{exam}

\begin{exam}\label{27} Let
$$f(\t):=\f_{27,3}(\t)=\frac{\eta(3\t)^3}{\eta(27\t)^3} +3=q^{-3}\prod_{n=1}^\infty (1-q^n)^{c(n)}\in M_0^{\sharp}(27).$$ 
Suppose $z_1$, $z_2$ and $z_3$ are zeros of $f(\t)$ in the fundamental domain $\mathcal F_{27}$, then the values $\f_{27,n}(z_1)$, $\f_{27,n}(z_2)$ and $\f_{27,n}(z_3)$ for all $n\geq 2$ satisfy
\begin{equation}\label{val27}
\f_{27,n}(z_1)+\f_{27,n}(z_2)+\f_{27,n}(z_3)=\sum_{d|n}c(d) d.
\end{equation}
Furthermore, $\f_{27,n}(z_1)$, $\f_{27,n}(z_2)$ and $\f_{27,n}(z_3)$ for all $n\geq 2$ are values in the cubic field $\Q(\sqrt[3]{9})$.

\end{exam}
\begin{proof}
The following is the list of the first few basis elements of $M_0^{\sharp}(27)$.

$$\begin{array}{rcl}
\f_{27,2}(\t)&=&q^{-2}+q+2q^4-q^7+{q}^{10}-{q}^{13}+{q}^{16}-3{q}^{19}-2{q}^{22}+3{q}^{25}+3{q}^{28}+\cdots \\
\f_{27,3}(\t)&=&q^{-3}+5q^6-7{q}^{15}+3{q}^{24}+15{q}^{33}-32{q}^{42}+9{q}^{51}+58{q}^{60}-96{q}^{69}+\cdots \\
\f_{27,4}(\t)&=&q^{-4}+2q^{-1}+5q^2+2q^5+4q^8-4{q}^{11}+5{q}^{14}-10{q}^{17}-3{q}^{20}-14{q}^{23}+13{q}^{26}+\cdots \\
\f_{27,5}(\t)&=&q^{-5}+q+2q^4+7q^7+8{q}^{10}-10{q}^{13}-6{q}^{16}-18{q}^{19}+20{q}^{22}-19{q}^{25}+4{q}^{28}+\cdots  \\
\f_{27,6}(\t)&=&q^{-6}+10q^3+11{q}^{12}-64{q}^{21}+109{q}^{30}+44{q}^{39}-503{q}^{48}+744{q}^{57}+295{q}^{66}-\cdots \\
 &\vdots&
\end{array}$$

Applying Proposition \ref{rec} to $q^3f=q^3\f_{27,3}=1+5q^9-7q^{18}+\cdots=\prod_{n=1}^\infty (1-q^n)^{c(n)}$, we compute $c(1)=c(2)=\cdots=c(8)=0$ and $c(9)=-5$. Thus $f=q^{-3}(1-q^9)^{-5}\cdots.$
Then by \eqref{divexp}, we find that
$$f_\theta(\t)=-3-\sum_{n=1}\sum_{d|n}c(d) dq^n=-3+45q^9+\cdots.$$

We note that  $f$ is non-vanishing at each cusp $\mathfrak s_i$ for $1\leq i\leq 5$ by the Iseki's transformation formula \cite[Theorem 5.8.1]{CS}.
Since $f$ has neither zero nor pole at each cusp $\mathfrak s_i$ for $1\leq i\leq 5$, $ord_{{\mathfrak s}_i}(f)=0$. Hence by \eqref{cons}, $E_f(\t)=0$. Since $X_0(27)$ has no elliptic points by \cite[Corollary 3.7.2]{DS}, it follows from \eqref{gtfne} that

\begin{equation}\label{}
{F}_{27,z_1}(\t)+{F}_{27,z_2}(\t)+{F}_{27,z_3}(\t)=-f_\theta(\t)-c(1)\fg_{27,-1}(\t)=-f_\theta(\t),
\end{equation}
because $c(1)=0$.

Also, from \eqref{divf}, we obtain \begin{equation}\label{val271}
\f_{27,n}(z_1)+\f_{27,n}(z_2)+\f_{27,n}(z_3)=c(1)a_{27}(n,-1)+\sum_{d|n}c(d) d,
\end{equation}
which proves \eqref{val27}. 
The first non-zero value of this occurs when $n=9$ as
$$\f_{27,9}(z_1)+\f_{27,9}(z_2)+\f_{27,9}(z_3)=-45.$$

Let $x_i=\f_{27,2}(z_i)$ and $y_i=\f_{27,3}(z_i)$ for $i=1,2,3$. Of course, $y_1=y_2=y_3=0$.
Then since $\f_{27,4}=\f_{27,2}^2$, $\f_{27,5}=\f_{27,2}\f_{27,3}-\f_{27,2}$, and $\f_{27,6}=\f_{27,3}^2$, we obtain from \eqref{val27} that
$$\begin{array}{rl}
x_1+x_2+x_3&=0,\\
y_1+y_2+y_3&=0,\\
x_1^2+x_2^2+x_3^2&=0,\\
(x_1y_1-x_1)+(x_2y_2-x_2)+(x_3y_3-x_3)&=0,\\
y_1^2+y_2^2+y_3^2&=0.
\end{array}$$
This along with the relation $\f_{27,2}^3=\f_{27,3}^2+3\f_{27,3}+9$ yields that $x_1x_2+x_2x_3+x_1x_3=0$ and $x_1x_2x_3=9$, and hence the values $\f_{27,2}(z_1)$, $\f_{27,2}(z_2)$, and $\f_{27,2}(z_3)$ are zeros of $X^3-9$. As each $\f_{27,n}(\t)$ is generated by  $\f_{27,2}(\t)$ and $\f_{27,3}(\t)$, all values  $\f_{27,n}(z_i)$ ($i=1,2,3$ and $n\geq 2$) lie in a cubic field.
\end{proof}

\begin{exam}\label{trace}
Let  $S$ be a subset of the exact divisors of $N$ and $\G=N+S$ denote the subgroup of ${\rm PSL}_2(\R)$ generated by $\G_0(N)$ and the Atkin-Lehner involutions $W_{Q,N}$ for all $Q\in S$.  
For a positive integer $D$ that is congruent to a square modulo $4N$, we denote by $\mathcal Q_{D,N}$ the set of positive definite binary quadratic forms 
$$
Q=[a,b,c]=aX^2+bXY+cY^2 \, \, (a,b,c\in \mathbb Z, N|a)
$$
of discriminant $D$, with the usual action of the group $\G$. To each $Q\in \mathcal Q_{D,N}$, we associate its unique root $\alpha_Q\in \mathbb H$, called a {\it CM point}.
Assume $\G$ is of genus zero, and let $j_\G$ denote the corresponding Hauptmodul. Then
$$
f(\t)=\prod_{Q\in \mathcal Q_{D,N}/\G}(j_\G(\t)-j_\G(\alpha_Q))
$$
is a meromorphic modular form of weight $0$ on $\G$ with a {\it Heegner divisor}, that is, a form whose zeros and poles are all at CM points and cusps. Hence by Borcherds' isomorphism, there is a certain weight $1/2$ weakly holomorphic modular form $g$ with Fourier expansion $g(\t)=\sum_{n\geq n_0}A(n)q^n$ satisfying  $f=q^h\prod_{n=1}^\i(1-q^n)^{A(n^2)}$ with $h=-| \mathcal Q_{D,N}/\G|$.

We note that if we consider $f$ as a meromorphic modular form on $\G_0(N)$, then the left-hand side of \eqref{divf} is the $D$-th modular trace of $\f_{N,m}$, i.e. 
$$
{\rm MT}(D, \f_{N,m}):=\sum_{Q\in \mathcal Q_{D,N}/\G_0(N)} e_{N,\alpha_Q} \f_{N,m}(\alpha_Q).
$$
Therefore, by applying Theorem \ref{d} to $f$, we can express the modular trace of $\f_{N,m}$ in terms of $A(n)$ and a coefficient of a certain weight 2 modular form in the Eisenstein space on $\G_0(N)$.

For a specific example, we let $\G$ be the group generated by $\G_0(N)$ and the Fricke involution $W_N:=\sm 0&-1\\ N&0\esm$ and consider a function discussed in \cite[Theorem 1.4]{CK}, which is a weakly holomorphic modular function on $\G$. Again, $\G$ is of genus zero and $j_N^+$ denotes the corresponding Hauptmodul. 

We define the Hurwitz-Kronecker class number $H(D)$ and a class number $H_N^+(D)$  by
$$H(D)=\sum_{Q\in \mathcal Q_{D,1}/\G(1)} e_{1,\alpha_Q}
\quad
{\rm and}\quad  
H_N^+(D)=\sum_{Q\in \mathcal Q_{D,N}/\G} e_{\G,\alpha_Q},$$
respectively.
Here $1/e_{\G,z}$ is the cardinality of $
\G_{z}/\{\pm1\}$ for each $z \in \H$, where $\G_{z}$ denotes the stabilizer of $z$ in $\G$. 
 For each cusp $\frak s$ in $S:=S_N-\{0,\i \}$, define $k_\frak s:=(v,N/v)$ if $\frak s$ is represented by a rational number $e/v$ such that $v$ is a divisor of $N$ and $e$ is coprime to $v$. Then for each positive integer $\ell$ such that $-\ell$ is congruent to a square modulo $4N$, the following holds: 
\begin{align*}
f(\t):= &
\prod_{\frak s\in S } (j_N^+(\t)-j_N^+(\frak s))^{-\frac12k_{\frak s}H(\ell/k_{\frak s}^2) }
\prod_{Q\in \mathcal Q_{\ell,N}/\G}(j_N^+(\t)-j_N^+(\alpha_Q))^{e_{\G,\alpha_Q}} \\
= &q^{-H_N^+(\ell)+\sum_{\frak s\in S}\frac12k_{\frak s}H(\ell/k_{\frak s}^2)}\prod_{\nu=1}^\i(1-q^\nu)^{-B(\nu^2,\ell)},
\end{align*}
where the class number $H(\ell/k_{\frak s}^2)$ is defined to be zero if $k_{\frak s}^2\nmid \ell$ and $B(\nu^2,\ell)$ is the coefficient of $q^\ell$ in a weakly holomorphic modular form of weight $3/2$ for $\G_0(4N)$ given in \cite[Eq.(12)]{CK}. Utilizing Theorem \ref{d} with some power $f^m$ of $f$ (we choose a suitable $m$ in order to have integer exponents in the definition of $f$), we obtain
$$
{\rm MT}(D, \f_{N,n})=-\sum_{d|n} B(d^2,\ell)d+\frac 1m \ve_{f^m}(n).
$$

Meanwhile, according to the theory of theta lifting developed by Bruinier and Funke \cite{BF},
${\rm MT}(D, \f_{N,n})$ can also be interpreted as a coefficient of $q^D$ of a certain harmonic weak Maass form of weight $3/2$. 
\end{exam}


\begin{thebibliography}{99}

\bibitem{AKN}
T.~Asai, M.~ Kaneko and H.~Ninomiya, \emph{Zeros of certain modular functions and an application}, Comment.~Math.~Univ.~St.~Pauli \textbf{46} (1997), 93--101.
%
%
\bibitem{BF}
J.~Bruinier and J.~Funke, \emph{On two geometric theta lifts}, Duke Math. J. \textbf{125} (2004), 45--90.
%
\bibitem{BKLOR}
K.~Bringmann, B.~Kane, S.~L\"{o}brich, K.~Ono and L.~Rolen,
\emph{On divisors of modular forms}, Adv.~Math. \textbf{329} (2018), 541--554.
%
\bibitem{BKO}
J.~Bruinier, W.~Kohnen and K.~Ono, \emph{The arithmetic of the values of modular functions and the divisors of modular forms}, Compos.~Math. \textbf{140} (2004), 552--566.
%
\bibitem{Choi}
D.~Choi, \emph{Poincar\'e series and the divisors of modular forms}, Proc.~Amer.~Math.~Soc. \textbf{138} (2010), 3393--3403.
%
\bibitem{CLL}
D.~Choi, M.~Lee and S.~Lim, \emph{Values of harmonic weak Maass forms on Hecke orbits}, J. Math.~Anal.~Appl. \textbf{477} (2019), 1046--1062.
%
\bibitem{CK}
S.~Choi and C.~H.~Kim, \emph{On the infinite product exponents of meromorphic modular forms for certain arithmetic groups}, Acta Arith. \textbf{145} (2010), 155--169.
%
\bibitem{CS}
H.~Cohen and F.~Str\"{o}mberg, \emph{Modular forms: A classical Approach}, Graduate Studies in Mathematics \textbf{179}, American Mathematical Society, Providence, RI, 2017.
%

\bibitem{Cox}
D.~A.~Cox, \emph{Primes of the form $x^2+ny^2$} (2nd Ed.)
Pure and Applied Mathematics (Hoboken),
John Wiley \& Sons, Inc., Hoboken, NJ, 2013.
%
%
\bibitem{DS}
F.~Diamond and J.~Shurman, \emph{A First Course in Modular Forms}, Graduate Texts in Mathematics \textbf{228}, Springer-Verlag, New York, 2005.
%
\bibitem{DJ}
W.~Duke and P.~Jenkins, \emph{On the zeros and coefficients of certain weakly holomorphic modular forms}, Pure Appl. Math. Q. \textbf{4}(4) (2008), Special Issue: In honor of Jean-Pierre Serre. Part 1, 1327–1340.
%
\bibitem{EH}
T.~Eguchi and K.~Hikami, \emph{Note on twisted elliptic genus of $K3$ surface}, Phys.~Lett. B, \textbf{694} (2011), 446--455.
%
\bibitem{GMR}
S.~Gun, M.~R. Murty and P. Rath, \emph{Algebraic independence of values of modular forms}, Int.~J.~Number Theory \textbf{7} (2011), 1065--1074.
%
\bibitem{JSS}
C.~Jennings-Shaffer and H.~Swisher, \emph{A note on the transcendence of zeros of a certain family of weakly holomorphic forms}, Int. J. Number Theory \textbf{10} (2014), 309–317.
%
\bibitem{JKK-hecke}
D.~Jeon, S.-Y.~Kang, and C.~H.~Kim, \emph{Hecke system of harmonic Maass functions and its applications to modular curves of higher genera}, (submitted for publication).
%
\bibitem{Kim}
C.~H.~Kim, \emph{Borcherds products associated with certain Thompson series}, 
Compos.~Math. \textbf{140} (2004), 541--551. 
%
\bibitem{KK}
 C.~H. Kim and J.~K. Koo, \emph{Generalization of Hauptmoduln of $\Gamma_1(N)$ by Weierstrass units and application to class fields}, Cent.~Eur.~J~.~Math. \textbf{9} (2011), 1389--1302.
%
\bibitem{Kohnen}
W.~Kohnen, \emph{Transcendence of zeros of Eisenstein series and other modular functions}. Comment. Math. Univ. St. Pauli \textbf{52} (2003), 55--57.
%
\bibitem{L}
 J.-S.~Lee, \emph{The restricted solutions of $ax+by=gcd(a,b)$}, Taiwanese J. Math. \textbf{12} (2008), 1191--1199.
%
%
\bibitem{RSD}
F.~K.~C~Rankin and H.~P.~F.~Swinnerton-Dyer, \emph{On the zeros of Eisenstein series,} Bull.~London Math. Soc. \textbf{2} (1970), 169--170.
%
\bibitem{Schneider}
T.~Schneider, \emph{Arithmetische Untersuchungen elliptischer Integrale (German)}, Math. Ann. \textbf{113}  (1937) 1--13.
%
\bibitem{Yang}
Y.~Yang, \emph{Defining equations of modular curves}, Adv. Math. \textbf{204} (2006), 481--508.
%
\end{thebibliography}
\end{document}